\documentclass[12pt,a4paper]{article}
\usepackage{amssymb,amsmath,amsthm}

\newcommand\cA{{\mathcal A}}

\newcommand\cD{{\mathcal D}}

\newcommand\cL{{\mathcal L}}

\newcommand\cP{{\mathcal P}}

\theoremstyle{plain}
\newtheorem{theorem}{Theorem}[section]
\newtheorem{lemma}[theorem]{Lemma}

\newtheorem{proposition}[theorem]{Proposition}

\theoremstyle{definition}

\newtheorem{claim}[theorem]{Claim}

\newtheorem*{remark}{Remark}

\newcommand\lref[1]{Lemma~\ref{lem:#1}}

\newcommand\cref[1]{Corollary~\ref{cor:#1}}

\newcommand\pref[1]{Proposition~\ref{prop:#1}}

\title{Towards a de Bruijn-Erd\H os theorem in the $L_1$-metric}

\author{
Ida Kantor\thanks{Computer Science Institute of Charles University,
Malostransk\'{e} n\'{a}m. 25, 118 00  Praha 1, Czech Republic.
 E-mail: ida@iuuk.mff.cuni.cz. Research supported by GA\v{C}R grant number P201/12/P288 and partially done while the first author visited the R\'{e}nyi Institute in 2011. This author
would like to acknowledge the generous support of the institute and
of the grant HNSF PD-83586 which funded this visit.} \and
Bal\'azs Patk\'os\thanks{Alfr\'ed R\'enyi Institute of Mathematics, P.O.B. 127, Budapest H-1364, Hungary. Email: patkos@renyi.hu. Research supported by
    Hungarian National Scientific Fund, grant number: PD-83586 and the J\'anos Bolyai Research Scholarship of the Hungarian Academy of Sciences.}
}

\begin{document}

\maketitle
\begin{abstract}
A well-known theorem of de Bruijn and Erd\H{o}s states that any set of $n$ non-collinear points in the plane determines at least $n$ lines. Chen and Chv\'{a}tal asked whether an analogous statement holds within the framework of finite metric spaces, with lines defined using the notion of {\em betweenness}. 

In this paper, we prove that the answer is affirmative for sets of $n$ points in the plane with the $L_1$ metric, provided that no two points share their $x$- or $y$-coordinate. In this case, either there is a line that contains all $n$ points, or $X$ induces at least $n$ distinct lines.

If points of $X$ are allowed to share their coordinates, then either there is a line that contains all $n$ points, or $X$ induces at least $n/37$ distinct lines.
\end{abstract}

\section{Lines in metric spaces}
Two well-known results are known under the name ``de Bruijn--Erd\H{o}s theorem''. One of them, published in~\cite{deBE48} in 1948, states that every set of $n$ points in the plane is either collinear or it determines at least $n$ distinct lines. 

The notion of a {\em line} can be extended naturally into an arbitrary metric space. 
If $(V,\rho)$ is an arbitrary metric space and $a,b,x\in V$, we say that $x$ is {\em between} the points $a$ and $b$ if 
\[
 \rho(a,b)=\rho(a,x)+\rho(x,b).
\]
Following the convention (established by Menger~\cite{Menger28})
of writing $[axb]$ to denote that $x$ lies between $a$ and $b$, the {\em line determined by the points $a$ and $b$}, denoted $\langle a,b\rangle$, consists of points in the set
\begin{equation}\label{def:line}
 \{x;[xab]\} \cup \{a\} \cup \{x;[axb]\} \cup \{b\} \cup \{x;[abx]\}.  
\end{equation}

We say that $a,b$ are the {\em defining points} of this line.

The situation is now slightly more complicated than with ordinary lines. For instance, it can happen that one line is a proper subset of another line. 

Chen and Chv\'{a}tal asked in~\cite{ChenChvatal08}
whether a statement analogous to the de Bruijn--Erd\H{o}s theorem holds in this setting as well. More precisely, they asked whether the following statement is true:

\begin{center}
{\em In an arbitrary metric space on $n$ points, either there are points $p,q$ such that the line $\langle p,q\rangle$ contains all points, or there are at least $n$ distinct lines.}
\end{center}

Despite many partial results, this question remains open. 
In this paper, we concentrate on a scenario that resembles the original de Bruijn--Erd\H{o}s situation. The points lie in the plane, but the usual Euclidean metric is replaced with the~$L_1$, or Manhattan, metric, defined by $d((u_1,u_2),(v_1,v_2))=|u_1-v_1|+|u_2-v_2|$. 

We encounter two very different kinds of lines in this case: those determined by two points that share their $x$- or $y$-coordinate, and those determined by two points that differ in both coordinates. If a set contains pairs of points that share a coordinate, we call it {\em degenerate}, otherwise it is {\em non-degenerate}.

Our first main result is that the answer to the Chen--Chv\'{a}tal question is affirmative for non-degenerate sets.
\begin{theorem}
\label{thm:main1}
In any metric space $(X,d)$ where $X$ is a non-degenerate finite set of points in the plane and $d$ is the $L_1$ metric restricted to $X$, either there is a line that contains all points of $X$, or there are at least 
$n$ distinct lines. 
\end{theorem}
In the the special case of degenerate sets we can prove a linear lower bound on the number of lines.
\begin{theorem}
\label{thm:main2}
In any metric space $(X,d)$ where $X$ is a finite set of points in the plane and $d$ is the $L_1$ metric restricted to $X$, either there is a line that contains all points of $X$, or there are at least 
$n/37$ distinct lines.
\end{theorem}

It is easy to see that the de Bruijn--Erd\H{o}s theorem is tight: consider $n-1$ points on a line, and one additional point that does not belong to this line. 
A similar construction 
works in the $L_1$ case as well to show that Theorem~\ref{thm:main1} is tight. Many other similarities between lines in $L_2$ and $L_1$ metrics in the plane can be found, but in many important aspects the situations are fundamentally different. 

Another important metric is $L_{\infty}$, defined by 
\[d_{\infty}((u_1,\dots,u_k),(v_1,\dots,v_k))=\max_{1\leq i\leq k}{|u_i-v_i|}.\]

Let $\phi$ denote the rotation around the origin by 45 degrees. Inspecting the $L_{\infty}$ lines, we see that for any pair of points  $v_1, v_2$, we have $\langle\phi(v_1),\phi(v_2)\rangle_\infty=\{\phi(u): u \in \langle v_1,v_2\rangle_1\}$. That is, for any finite set of points in the plane, the $L_{\infty}$-lines correspond to the $L_1$-lines of the rotated set (we will describe what $L_1$-lines look like in Section~\ref{sec:monot}). 
The following theorem is then a direct corollary of Theorems~\ref{thm:main1} and~\ref{thm:main2}.

\begin{theorem}
\label{thm:linfty}
In any metric space $(X,d)$ where $X$ is a finite set of points in the plane and $d$ is the $L_{\infty}$ metric restricted to $X$, either there is a line that contains all points of $X$, or there are at least 
$n/37$ distinct lines.
If, moreover, we have $|u_1-v_1|\neq|u_2-v_2|$ for every two points $(u_1,u_2)$, $(v_1,v_2)$ of $X$, then there are at least $n$ distinct lines.
\end{theorem}

To close this section, let us list the most important results that are known regarding the Chen--Chv\'{a}tal question. In~\cite{ChenChvatal08}, Chen and Chv\'{a}tal proved that in every metric space on $n$ points, either some line contains all points, or there are at least $\lg n$ distinct lines. Chiniforooshan and Chv\'{a}tal proved in~\cite{ChiniChvatal11} that 
\begin{itemize}
 \item in every metric space induced by a connected graph, either there is a line containing all points, or there are $\Omega(n^{2/7})$ distinct lines,
\item in every metric space on $n$ points, there are $\Omega((n/\rho)^{2/3})$ distinct lines, where $\rho$ is the ratio between the larges distance and the smallest nonzero distance, and
\item in every metric space on $n$ points where every nonzero distance equals 1 or 2, there are $\Omega(n^{4/3})$ lines.
\end{itemize}
The last result implies an affirmative answer to the Chen--Chv\'{a}tal question for metric spaces with nonzero distances equal to 1 or 2, as long as $n$ is large enough. In~\cite{Chvatal2}, Chv\'{a}tal proved a variation on this result---he showed that the answer is affirmative  for such metric spaces with any value of $n$. The authors of~\cite{BBCCCC} proved that this is the case also for metric spaces induced by connected chordal graphs. In~\cite{BBCCC2}, the same authors generalize the de Bruijn--Erd\H{o}s theorem in a different direction.
Several results dealing with lines in metric spaces induced by graphs were also proved by Jir\'{a}sek and Klav\'{i}k in~\cite{jirasekKlavik}. 

As an example of a result that does not hold in this setting, let us mention the Sylvester--Gallai theorem, conjectured by Sylvester in~\cite{Sylvester} and proved by Gallai many years later (see~\cite{Erdos1982} for a history of the problem). This classic theorem of ordered geometry states that every set of points in the plane is either collinear, or it contains two points such that the line passing through them contains no other points of the set. De Bruijn and Erd\H{o}s observed in~\cite{deBE48} that their theorem follows from this result. In~\cite{Chvatal04}, Chv\'{a}tal provided an example showing that the Sylvester--Gallai theorem is no longer necessarily true in general metric spaces. He conjectured, however, that if we consider the line $\langle a,b\rangle$ to be the recursive closure of~(\ref{def:line}) instead of~(\ref{def:line}) itself, then the statements holds in arbitrary finite metric space. This conjecture was verified by Chen in~\cite{Chen76}.

In the rest of this paper, we will assume the hypotheses of Theorem~\ref{thm:main2}, i.e., $X$ will be a finite set of points in the plane equipped with the $L_1$ metric.

\section{Monotone sequences of points}\label{sec:monot}
In this section, we will see that long increasing or decreasing sequences of points guarantee the existence of many lines. 

We will denote points by lowercase letters $p,q,\dots$ or by pairs of coordinates, whatever is more convenient at the moment.


We say that a pair of points $(x_1,y_1)$, $(x_2,y_2)$ is  
\begin{itemize}
\item
\textit{increasing} if $(x_1-x_2)(y_1-y_2)>0$, 
\item
\textit{decreasing} if $ (x_1-x_2)(y_1-y_2)<0$,
\item
\textit{horizontal} if $y_1=y_2$,
\item
\textit{vertical} if $x_1=x_2$.
\end{itemize}


If $(x_1,y_1),(x_2,y_2)$ is an increasing pair with $x_1<x_2$, the line determined by these points consists of points in the (closed) rectangle determined by the two points and of points lying in the two quarter-planes $\{(x,y); x\leq x_1, y\leq y_1\}$ and $\{(x,y); x\geq x_2, y\geq y_2\}$. The line determined by a horizontal pair $(x_1,y_1),(x_2,y_1)$ consists of points $(x,y_1)$ with $x_1\leq x\leq x_2$ and of points in the two half-planes $\{(x,y);  x\leq x_1\}$ and $\{(x,y);  x\geq x_2\}$. 
 Lines determined by decreasing or vertical pairs look similar.

We say that $p_1,p_2,...,p_k$ form an {\em increasing sequence} if every pair of points from the sequence is increasing. {\em Decreasing sequences} are defined analogously.

Let us start with two easy observations.
\begin{proposition}
\label{prop:mon} If $p_1,p_2$ is an increasing pair and $p \notin \langle p_1,p_2\rangle$, then either $p,p_1$ or $p,p_2$ is a decreasing pair. Also, if $p_1,...,p_k$ is an increasing sequence and the pair $p,p_j$ is decreasing, then $p_i \notin \langle p,p_j\rangle$ for all $i \ne j$.
\end{proposition}

\begin{lemma}
\label{lem:monseqline} If $X$ is a non-collinear set that contains an increasing sequence of size $k$, then $X$ defines at least $k$ different lines.
\end{lemma}

\begin{proof}
Let $p_1,...,p_k$ be the points of the increasing sequence and let us consider the pair $p_1,p_2$. As there is a point $p \notin \langle p_1,p_2\rangle$, we obtain by the first part of \pref{mon} that $p,p_i$ form a decreasing pair with $i=1$ or $i=2$ and therefore, by the second part of \pref{mon}, $\langle p,p_i\rangle$ contains only $p_i$ from the increasing sequence. Now we repeat the procedure for $p_{3-i}, p_3$ to obtain another line $\ell_2$ that contains either only $p_{3-i}$ or $p_3$ among the points of the increasing sequence. Continuing in a similar way, we finally obtain $k-1$ lines that contain exactly one point from the increasing sequence and these points are pairwise different, thus the lines are different. Furthermore for any $i,j$ with 
$1\le i$ and $j\le k$, the line $\langle p_i,p_j\rangle$ contains all points of the increasing sequence, therefore there is a $k$-th different line defined by $\cP$.
\end{proof}

A version of Lemma~\ref{lem:monseqline} for decreasing sequences can be proved in an analogous way. 

The well-known Erd\H{o}s--Szekeres theorem~\cite{{ESZ}} implies that whenever we have a set of $n$ points in the plane such that no two share their $x$- or $y$-coordinates, $I$ is the number of points in the longest increasing sequence and $D$ is the number of points in the longest decreasing sequence, then $D\cdot I\geq n$. Together with  Lemma~\ref{lem:monseqline} this implies the existence of at least $\sqrt{n}$ distinct lines in such sets. 

Let us take this idea a step further and prove that $\Omega(n^{2/3})$ lines exist in such sets. Draw a vertical line so that exactly half of the points are to the left of the line, and a horizontal line so that exactly half of the points are above it. We obtain four quadrants and either the top left and the bottom right quadrants have at least $n/4$ points each, or the same holds for the top right and the bottom left. Let us suppose the first option holds. If either the top left or the bottom right quadrant contains a decreasing sequence with $\Omega(n^{2/3})$ points, we are done by Lemma~\ref{lem:monseqline}. If this is not the case, by the Erd\H{o}s--Szekeres theorem we obtain an increasing sequence with $\Omega(n^{1/3})$ points in each of the two quadrants. We have $\Omega(n^{2/3})$ pairs of points such that one point is in one of the two sequences and the other point is in the second one. The lines determined by these pairs are all distinct. 

Much stronger results can be obtained by more complicated ideas along these lines, but in order to reach our ultimate goal and prove the existence of a linear number of lines, we have to make better use of the structure of the set $X$. We will do this in the subsequent sections.





\section{Definition of a partial order; many lines are found}\label{sec:many}

In this section, $X$ will be a non-collinear set of points in the plane, and $Y$ will be an $n$-element subset of $X$ in which no two points share their $x$-coordinates and no two points share their $y$-coordinates.

Let us define a partial order $\preceq$ on $Y$ by putting $p \preceq q$ whenever $p,q$ is an increasing pair and $p$ has smaller $x$-coordinate than $q$. Define a partition $Y=\cA_0\cup \dots \cup \cA_k$ recursively as follows. Let $\cA_0$ be the set of minimal elements of the partially ordered set $(Y,\preceq)$, and if $\cA_0,\dots,\cA_{i-1}$ are already defined, let $\cA_i$ be the set of minimal elements of the partially ordered set obtained from $(Y,\preceq)$ by removing all points that belong to $\cA_0,\dots,\cA_{i-1}$. We will call each such $\cA_i$ a {\em layer}.  
The aim of this section is to find for each $i>0$ a set of lines of cardinality $|\cA_i|$.  

We will say that two points are {\em neighbors} if they are adjacent in the cover graph of this partial order.

Let $G_i$ be the (bipartite) cover graph induced by $\cA_{i-1}$ and $\cA_i$ (that is, $V(G_i)=\cA_{i-1}\cup \cA_i$, and edges are all increasing pairs of points). Fix a positive integer $i$, and let $q_1,\dots,q_t$ be all elements of $\cA_i$ such that all their neighbors in $G_i$ have degree~1 in $G_i$. For each such $q_s$, select one of its neighbors (in $G_i$) arbitrarily, and call it $\phi(q_s)$. Define a set of lines $\cL_i$. This set contains
\begin{itemize}
\item  all lines $\langle \phi(q_{s_1}),\phi(q_{s_2})\rangle$ (these are {\em decreasing lines}) and 
 \item all lines $\langle p,q\rangle$ such that $p\in \cA_{i-1}$, $q\in \cA_i$, the pair $p,q$ is increasing, and\\ $\deg_{G_i}(p)>1$ (these are {\em increasing lines}).
\end{itemize}

The lines in $\cL_i$ have several useful properties. In order to accommodate lines that will be defined later, we formulate these properties in a slightly more general way here. 

\begin{claim}
Each decreasing line $\ell=\langle \phi(q_{s_1}),\phi(q_{s_2})\rangle \in \cL_i$ has the following properties:
\begin{itemize}
\item
$\mathfrak{D}_1$: it contains all points of $\cA_{i-1}$, 
\item
$\mathfrak{D}_2$: it contains all points of $\cA_i$ except for the neighbors (in $G_i$) of the defining points of $\ell$. 
(this is because $\phi(q_{s_i})$ forms an increasing pair with $q \in \cA_i$ if and only if $q=q_{s_i}$),
\item
$\mathfrak{D}_3$: whenever $p$ is a defining point of $\ell$ that belongs to $\cA_j$ for some index $j$, there exists an increasing sequence $p_0,...,p_{j}=p$  such that for all indices $s$ we have $p_s\in \cA_s$, and for $s\leq j-1$ we have $p_s \notin \ell$.
\end{itemize}
\end{claim}

\begin{claim}
Each increasing line $\ell=\langle p,q\rangle \in \cL_i$ has the following properties:
\begin{itemize}
\item
$\mathfrak{I}_1$: $\ell \cap \cA_{i-1}=\{p\}$ and $\ell \cap \cA_i= \{q\}$,
\item
$\mathfrak{I}_2$: there exists a point $q' \in \cA_i$ with $p\preceq q'$ and $q' \notin \ell$ (this is because $p$ by definition has at least two neighbors in $G_i$),
\item
$\mathfrak{I}_3$: for any $j\le i-1$, there exists a point $p_j \in \cA_j$ with $p_j \in \ell$,
\item
$\mathfrak{I}_4$: for any $j\ge i$, if $p_j \in \cA_j \cap \ell$, then for all $p_{j+1} \in \cA_{j+1}$ with $p_j \preceq p_{j+1}$ we have $p_{j+1} \in \ell$.
\end{itemize}
\end{claim}

Let $\cL$ be the union of all $\cL_i$ for $i>0$. For the purposes of the following claim, we consider $\cL$ to be a multiset. We will see later (in Lemma~\ref{lem:distinct1}) that the lines in $\cL$ are in fact distinct, i.e., it is a set.

\begin{claim}
 We have $|\cL|\geq \sum_{i>0} (|\cA_i|-1).$ 
\end{claim}
\begin{proof}
For a fixed $i$, suppose that $\cA_i$ contains $c_i$ points for which all neighbors in $G_i$ have degree equal to~1 and $d_i$ points that have neighbors of higher degree. Then $\cL_i$ contains at least $d_i$ increasing lines and $\binom{c_i}{2}$ decreasing lines. We have $\binom{c_i}{2}\geq c_i-1$ for all non-negative integers $c_i$, so we obtain $|\cL_i|\geq |\cA_i|-1$.
\end{proof}

Some of the sets $\cL_i$ might contain only $|\cA_i|-1$ lines. This can happen only if $c_i=1$ or $c_i=2$. In both cases we need to find a line $\ell_i^+$ different from all lines in $\cL$. We will put $\cL'_i=\cL_i \cup \{\ell_i^+\}$ and $\cL'=\cup_{i>0}\cL'_i$.

 If $c_i=1$, $\cL_i$ does not contain any decreasing line. If $|\cA_{i-1}|\ge 2$, pick any two points $p,q \in \cA_{i-1}$ and let $\ell^+_i=\langle p,q \rangle$. If $|\cA_{i-1}|=1$, then we also have $|\cA_i|=1$, since otherwise the only point in $\cA_{i-1}$ would have degree at least $2$, implying $c_i=0$. Let $q$ denote the only point in $\cA_i$ and let $r$ be a point outside $\langle p,q\rangle$. 
We have $\cA_j\subseteq \langle p,q\rangle$ for all $j\geq i-1$. It follows that if $r\in Y$, then $r \in\cup_{j< i-1}\cA_j$ and the pair $r,p$ is decreasing.
 If the line $\langle p,q\rangle$ contains all points of $Y$, then $r\in X\setminus Y$, and the pair $r,p$ can be increasing or decreasing. Let $\ell^+_i=\langle r,p\rangle$. We will call such line a {\em special} line. 

\begin{claim}
Decreasing lines defined in this step satisfy properties $\mathfrak{D}_1$, $\mathfrak{D}_2$, and $\mathfrak{D}_3$. A special increasing line satisfies $\mathfrak{I}_3$. The pair $r,q$ is decreasing since $r \notin \langle p,q \rangle$, so a special increasing line $\langle r,p\rangle$ also satisfies 
\begin{itemize}
\item
$\mathfrak{I}_2'$: no point of $\cA_i$ belongs to $\langle r,p\rangle$.
\end{itemize} 
\end{claim}

\vspace{3mm}

If $c_i=2$, there is exactly one decreasing line in $\cL_i$. Let $q_1=(x_1,y_1)$ and $q_2=(x_2,y_2)$ be the two points in $\cA_i$ that only have neighbors of degree~1 in $G_i$, and suppose $x_1<x_2$ and $y_1>y_2$. Let $p_1,p_2,...,p_k$ be the points of $\cA_{i-1}$ sorted in increasing order according to their $x$-coordinate. Note that the neighbors of $q_2$ in 
$G_i$ are consecutive points in the above ordering. Let $p_a$ be the neighbor of $q_2$ with the smallest index and define $\ell^+_i=\langle p_{a-1},q_2\rangle$. Note that $a>1$ holds, since all neighbors of $q_1$ have smaller indices than $a$. 


\begin{claim}
The line $\ell^+_i$  has the property $\mathfrak{D}_3$, as well as the following two properties:
\begin{itemize}
\item
$\mathfrak{D}_1'$: $\ell^+_i$ contains all points of $\cA_{i-1}$ except the neighbors of $q_2$, therefore $\ell^+_i$ contains at least one point from all increasing pairs from $\cA_{i-1}\cup \cA_i$, 
\item
$\mathfrak{D}_2'$: it contains all points of $\cA_i$ except for the neighbors of $p_{a-1}$.
\end{itemize}
\end{claim}

\begin{lemma}\label{lem:distinct1}
 All lines in $\cL'$ are distinct.
\end{lemma}

\begin{proof}

Since for each line in $\cL'=\cup_{i>0}\cL'_i$, at least one of its defining points lies in $Y$, the condition that no two points of $Y$ share their $x$- or $y$-coordinates guarantees that an increasing line never coincides with a decreasing line. 

Let $\ell_1 \in \cL'_i, \ell_2\in \cL'_j$ be two decreasing lines. If $i=j$ and both lines satisfy $\mathfrak{D}_2$, then they are distinct. Otherwise one of them satisfies $\mathfrak{D}_1$ and the other $\mathfrak{D}_1'$ and thus they are distinct.

Suppose $i<j$. The line $\ell_2$ has at least one defining point $u$ on $\cA_{j-1}$, and by $\cD_3$ there exists an increasing sequence  $u_0,\dots,u_{j-1}=u$ such that $u_s\not\in \ell_2$ for $s\leq j-2$. If $u_{i-1}\in \ell_1$, we are done. If not, $\ell_1$ satisfies $\cD_1'$. The vertex $u_{i-1}$ has degree~1 in $G_i$, so $v_i$ is the vertex $q_2$ defined above. If $i<j-1$, then $q_2\in \ell_1\setminus \ell_2$ and we are done. 
Now suppose $i=j-1$. We have $q_2=u$. If the other defining point of $\ell_2$ is on $\cA_{j-1}$, repeat the same argument to get $q_2=v$, a contradiction. If the other defining point of $\ell_2$ is on $\cA_j$, similar argument applies. 
If the other defining point of $\ell_2$ is on $\cA_k$ with $k<j-1$, then $|\cA_{j-1}|=1$, a contradiction with the assumption that according to $\cD_1'$ for $\ell_1$, it contains two different points, $q_1$ and $q_2$.  


Let $\ell_1 \in \cL'_i, \ell_2\in \cL'_j$ be two increasing lines. If $i=j$, then both $\ell_1$ and $\ell_2$ satisfy $\mathfrak{I}_1$ (if $\cL'_i$ contains a special increasing line, then there is only one increasing line in $\cL'_i$), so they are distinct. If $i< j$ and $\ell_1$ satisfies $\mathfrak{I}_2'$, then the lines are distinct since $\mathfrak{I}_3$ holds for $\ell_2$. Otherwise they are distinct because $\ell_1$ satisfies $\mathfrak{I}_4$, and $\ell_2$ satisfies $\mathfrak{I}_2$ or $\mathfrak{I}_2'$.
\end{proof}

For future reference, let us state also the following observation as a lemma.

\begin{lemma}\label{lem:size3}
 For $i>0$ we have $|\cL_i'|\geq |\cA_i|$.
\end{lemma}

\section{Linear lower bound for arbitrary sets of points}
\begin{lemma}
\label{lem:notwo1}
Let $X$ be a set of non-collinear points in the plane. Moreover suppose that $X$ has an $n$-element subset in which no two points share their $x$-coordinates and no two points share their $y$-coordinates. Then $X$ induces at least $n/2$ distinct lines. 
\end{lemma}

\begin{proof}
The $n$-element subset of $X$ satisfies the conditions placed on $Y$ in Section~\ref{sec:many}. Let us use the notation and results introduced in that section.
Since $|\cL'|=\sum_{i>0} |\cL_i'|$, Lemma~\ref{lem:size3} implies that 
\[
|\cL'|\ge \sum_{i>0}|\cA_i|.
\]
If this sum is at least $n/2$, then we have found $n/2$ distinct lines determined by $X$. Otherwise, since points in $\cA_0$ form a decreasing sequence, $X$ induces at least $|\cA_0|=n-\sum_{i>0}|\cA_i|>n/2$ distinct lines by \lref{monseqline}.
\end{proof}


\begin{lemma}
\label{lem:horiz} 
Let $\{(a,y_1),(a,y_2)\}$ and $\{(b,y_1'),(b,y_2')\}$ be two vertical pairs in $X$ such that the pairs $\{y_1,y_2\}$ and $\{y_1',y_2'\}$ are distinct. Then the lines $\langle (a,y_1),(a,y_2) \rangle$ and $\langle (b,y_1'),(b,y_2')\rangle$ are distinct.
\end{lemma}

\begin{proof}

Suppose first that $a=b$. 
If $p_1,...,p_k \in X$ are points with the same $x$-coordinate $a$ and with $y$-coordinates $y_1,...,y_k$, then for any integer $i$ with $1 \le i <k$, there exists a point $p'_i=(x'_i,y'_i) \in X$ such that $a\ne x'$ and $y_i <y'_i<y_{i+1}$. 
For $l<m$, the line $\langle p_l,p_m \rangle$ contains all the points $p'_i$ except for $p'_l,\dots,p'_{m-1}$. It follows that the lines $\langle p_l,p_m\rangle$ are pairwise distinct for all distinct pairs $\{p_l,p_m\}$.

If $a\neq b$
and $y_1<y_2\le y_1'<y_2'$, then for any point $p\notin \langle (a,y_1),(a,y_2)\rangle$ we have 
$p \in \langle (b,y_1'),(b,y_2')\rangle$. If $y_1<y_1'<y_2<y_2'$, then $(b,y_1') \notin \langle(a,y_1),(a,y_2)\rangle$. Finally, if $y_1<y_1'<y_2'<y_2$, then again $(b,y_1') \notin \langle(a,y_1),(a,y_2)\rangle$.
\end{proof}

\begin{lemma}\label{dist}
 Let $\{(a,y_1), (a,y_2)\}$ and $\{(b,y_1'), (b,y_2')\}$ be two distinct vertical pairs of points in $X$ such that there exists a point $(a,y_3)\in X$ with $y_3$ strictly between $y_1$ and $y_2$.
 Then the line $\ell_1=\langle (a,y_1), (a,y_2)\rangle$ does not coincide with the line $\ell_2=\langle(b,y_1'),(b,y_2')\rangle$. 
\end{lemma}

\begin{proof}
 If $\{y_1,y_2\}=\{y_1',y_2'\}$, then $a\neq b$ and $(a,y_3)\in \ell_1\setminus \ell_2$. If the two pairs of $y$-coordinates do not coincide, then the lines are distinct by Lemma~\ref{lem:horiz} 
\end{proof}

Analogous pair of lemmas holds for horizontal lines.

\begin{lemma}\label{vert}
 Let $X$ be a set of $n$ points such that each point shares its $x$-coordinate with at least four others. Then $X$ determines at least $n$ (vertical) lines.
\end{lemma}

\begin{proof}
 A set of $m$ points that all share the same $x$-coordinate determines at least $\binom{m-1}{2}$ non-consecutive pairs of points. If $m\geq 5$, this is greater than $m$. Overall we have at least $n$ non-consecutive vertical pairs. These determine distinct lines by Lemma~\ref{dist}.  
\end{proof}


\begin{proof}[Proof of Theorem~\ref{thm:main2}.]
 Let $0\leq c,d\leq 1$ be some constants. If at least $cn$ points share their $x$-coordinates with at least four other points, then we have $cn$ vertical lines by Lemma~\ref{vert}. If not, we have at least $(1-c)n$ points that share their $x$-coordinate with at most three other points. Deleting some points we get a set of at least $(1-c)n/4$ points with unique $x$-coordinates. Again, if at least $d$-fraction of these points share their $y$-coordinates with at least four other points, then we have $(1-c)dn/4$ horizontal lines. If not, deleting some points we get a set of at least $(1-d)(1-c)n/16$ points such that no two share $x$- or $y$-coordinate. By Lemma~\ref{lem:notwo1} we then obtain at least $(1-d)(1-c)n/(16\cdot 2)$ lines. Set $d=1/9$, $c=1/37
$.  
\end{proof}

Since our method does not yield the conjectured lower bound of $n$, we did not attempt to improve the multiplicative constant $1/37$.

\section{Additional lines for non-degenerate sets}

If $X=Y$, we can prove a stronger lower bound on the number of lines.
In this section, we will assume that $X$ is a non-collinear set of $n$ points in the plane, 
such that no two points share their $x$- or $y$-coordinate.

Define the set $\cL'$ of lines as in Section~\ref{sec:many}.
Note that in this case all lines $\ell^+_i$ are decreasing.
We need to add a set $\cL_0$ of new lines so that $|\cL_0 \cup \cL'_1| \ge |\cA_0 \cup \cA_1|$ holds. 

Partition $\cA_0$ into four parts $A_{0,0}, A_{0,h}, A_{0,-},A_{0,+}$. A point $p \in \cA_0$ belongs to $A_{0,0}$ if 
there exists no $q$ such that $p,q$ is an increasing pair. 
A point $p \in \cA_0\setminus A_{0,0}$ belongs to $A_{0,h}$ if 
it belongs to some increasing pair but all pairs $p,q$ with $q\in \cA_1$ are decreasing. 
$A_{0,-}$ consists of the points of $\cA_0$ that have degree~1 in $G_1$, and $A_{0,+}$ is set of points in $\cA_0$ with degree at least 2 in $G_1$. 

Let us first introduce the new increasing lines. For any point $p \in \cA_0 \setminus A_{0,0}$ let $h(p)$ denote the minimum index $i$ such that there exists a point $q \in \cA_i$ with $p \prec q$. We define
\[
\cL_{0,incr}=\{\langle p,q\rangle: p \in A_{0,-} \cup A_{0,h} ,q \in \cA_{h(p)}, p\prec q\}.
\]

Every line $\ell=\langle p,q\rangle \in \cL_{0,incr}$ has the following properties:

\begin{itemize}
\item
$\mathfrak{I}_1''$: $\ell \cap \cA_0=\{p\}, \ell \cap \cA_{h(p)}=q$, and for all $0 <j<h(p)$ we have $\ell \cap \cA_j=\emptyset$,
\item
$\mathfrak{I}_4''$: for any $j\ge h(p)$, if $p_j \in \cA_j \cap \ell$, then for all $p_{j+1} \in \cA_{j+1}$ with $p_j \prec p_{j+1}$ we have $p_{j+1} \in \ell$.
\end{itemize}

\begin{lemma} 
The lines in $\cL'\cup \cL_{0,incr}$ are all distinct.
\end{lemma}

\begin{proof}
Property $\mathfrak{I}_1''$ ensures that two lines $\langle p,q\rangle$ and $\langle p',q'\rangle$ that both belong to $\cL_{0,incr}$ are distinct unless $\{p,q\}=\{p',q'\}$. A line $\langle p,q\rangle=\ell_1 \in \cL_{0,incr}$ is different from any increasing line $\ell_2 \in \cup_{i=1}^{h(p)}\cL_i$ because $\ell_1$ satisfies $\mathfrak{I}_1''$ and $\ell_2$ satisfies $\mathfrak{I}_1$. If $\ell_2 \in \cup_{i>h(p)}\cL_i$, then $\ell_1 \neq \ell_2$ because $\ell_1$ satisfies $\mathfrak{I}_4''$ and $\ell_2$ satisfies $\mathfrak{I}_2$. 
\end{proof}



To introduce the new decreasing lines, let $A_{1,c}$ denote the set of all elements of $\cA_1$ such that all their neighbors in $G_1$ have degree~1 in $G_1$ (i.e., $|A_{1,c}|=c_1$). Writing $A_{1,d}=\cA_1 \setminus A_{1,c}$, consider the induced subgraph $G'=G_1[A_{1,d},A_{0,+}]$. Let us create a final partition of $A_{1,d}$ and $A_{0,+}$. Let $A$ consist of the points of $A_{1,d}$ that have degree~1 in $G'$ and let $C \subseteq A_{0,+}$ be the set of neighbors of $A$ in $G'$. Write $B=A_{1,d}\setminus A$ and $D=A_{0,+}\setminus C$. Let $p_1,p_2,...,p_s$ be the enumeration of $C$ in increasing order according to their $x$-coordinates and let $q_k^1,q_k^2,...,q_k^{m(k)}$ be the neighbors of $p_k$ in $A$ enumerated again in increasing order according to their $x$-coordinates. Let $u_k$ be either the rightmost point in $A_{0,-} \cup A_{0,+}$ that has larger $y$-coordinate than $q_k^1$ or the leftmost point in $A_{0,-} \cup A_{0,+}$ that has larger $x$-coordinate than $q_k^{m(k)}$. If $k \ge 2$ then $p_{k-1}$ has larger $y$-coordinate than $q_k^1$ while if $k \le s-1$, then $p_{k+1}$ has larger $x$-coordinate than $q_k^{m(k)}$. If this was not the case, $p_{k-1}$ and $q_k^1$ (or $p_{k+1}$ and $q_k^{m(k)}$) would form an increasing pair, contradicting that $q_k^1$ and $q_k^{m(k)}$ have degree~1 in $G'$. It follows that if $s \ge 2$,then $u_k$ exists for all $1\le k \le s$. If $s \ge 2$, define
\[
\cL_{0,decr}=\{\langle u_k,q_k^1\rangle: 1 \le k\le s\}.
\]

\begin{claim}
 Every line $\ell=\langle u_k,q_k^1\rangle \in \cL_{0,decr}$ has the following properties:
\begin{itemize}
\item
$\mathfrak{D}_1''$: $\ell$ contains all points of $\cA_0$ except for $p_k$ (this is because the only point in $\cA_0$ with which $q_k^1$ forms an increasing pair is $p_k$),
\item
$\mathfrak{D}_2''$: there exists a point $q \in \cA_1$ with $q \notin \ell$ (since $u_k\in A_{0,-} \cup A_{0,+}$, it has positive degree in $G_1$), 
\item
$\mathfrak{D}_3''$: for every $q \in \cA_1$ either $q \in \ell$ holds or there exists $p\in\cA_0$ with $p \prec q$ and $p\in \ell$ (if the only neighbor of $q$ is $p_k$, then by definition $q \in \ell$ holds, otherwise we are done by $\mathfrak{D}_1''$).
\end{itemize}
\end{claim}

\begin{lemma} 
The lines in $\cL'\cup \cL_{0,decr}$ are all distinct.
\end{lemma}

\begin{proof}
For $k_1\neq k_2$, the lines $\langle u_{k_1},q_{k_1}^1\rangle$ and $\langle u_{k_2},q_{k_2}^1\rangle$ are distinct because of the property $\mathfrak{D}_1''$. If $\ell_1=\langle u_k,q_k^1\rangle$ and $\ell_2 \in\cL'_1$, then they are distinct because $\ell_1$ satisfies $\mathfrak{D}_1''$ and $\ell_2$ satisfies $\mathfrak{D}_1$ or $\mathfrak{D}_1'$, since we have $p_k \notin A_{0,-}$. 
If $\ell_1=\langle u_k,q_k^1\rangle$ and $\ell_2 \in\cL'_2$, the property $\mathfrak{D}_2''$ guarantees that there is a point $q\in \cA_1$ that does not belong to $\ell_1$. If $q\in \ell_2$, we are done. 
If this is not the case, then $\ell_2$ satisfies $\mathfrak{D}_1'$, so $q$ has a neighbor in $\cA_2$ that belongs to $\ell_2$. This point does not belong to $\ell_1$.
Finally, if $\ell_1=\langle u_k,q_k^1\rangle$ and $\ell_2 \in\cL'_j$ with $j \ge 3$, then they are distinct because $\ell_1$ satisfies $\mathfrak{D}_3''$ and $\ell_2$ satisfies $\mathfrak{D}_3$.
\end{proof}

Thanks to the assumption that no two points in $X$ share their $x$- or $y$-coordinate, an increasing line never coincides with a decreasing line. We have therefore shown that the lines in $\cL' \cup \cL_{0,decr} \cup \cL_{0,incr}$ are all distinct.

\begin{lemma}\label{size2}
 We have $|\cL'_1 \cup \cL_{0,decr} \cup \cL_{0,incr}| \ge |\cA_0 \cup \cA_1|-2$.
\end{lemma}

\begin{proof}We partitioned $\cA_1$ into $A_{1,c}, A$, $B$, while $\cA_0$ is partitioned into $A_{0,0}, A_{0,h}, A_{0,-}, C,D$. The number of decreasing lines in $\cL'_1$ is at least $c_1=|A_{1,c}|$. The number of increasing lines in $\cL'_1$ is $e(G')$, where $G'=G_1[A_{1,d}, A_{0,+}]=G_1[A \cup B, C\cup D]$. 

Let us note that we have $|\cL_{0,incr}| \ge |A_{0,-}| +|A_{0,h}|$ and $|A_{0,0}| \le 1$. 
The first inequality follows from the definition of $\cL_{0,incr}$. For the second inequality observe that if $p_1, p_2$ belonged to $A_{0,0}$, then the decreasing line $\langle p_1,p_2\rangle$ would contain all points of $X$. 
The number of lines in $\cL_{0,decr}$ is $|C|$ if $|C| \neq 1$ and $0=|C|-1$ if $|C|=1$. 

We claim that $e(G') \ge |A| +|B|+|D|$. Vertices in $A$, by definition, have degree~1 in $G_1$, thus the number of edges adjacent to $A$ is exactly $|A|$ and these edges have their other endpoint in $C$. The degree of all vertices of $B \cup D$ is at least 2, thus the number of edges in $G'$ not adjacent to vertices in $A$ is at least $|B|+|D|$.
Adding all these inequalities we obtain that $|\cL'_1 \cup \cL_{0,decr} \cup \cL_{0,incr}| \ge |\cA_0 \cup \cA_1|-2$.
\end{proof}


\section{Non-degenerate sets: $n$ lines exist}
If $X=Y$, then we can improve Lemma~\ref{lem:notwo1} and obtain Theorem~\ref{thm:main1}.

\begin{proof}[Proof of Theorem~\ref{thm:main1}]
Putting together Lemma~\ref{lem:size3} with Lemma~\ref{size2}, we see that $X$ induces at least $n-2$ lines, and moreover the term $-2$ would disappear if we had $A_{0,0}=\emptyset$ and $|C|\neq 1$ in the proof of the Lemma~\ref{size2}. 

To overcome these last two problems, let us introduce three more partial orders on $X$. Let $(X,\preceq_1)=(X,\preceq)$, and define $(X,\preceq_2),(X,\preceq_3),(X,\preceq_4)$ by 
\begin{itemize}
\item
$p \preceq_2 q$ if and only if $q \preceq_1 p$,
\item
$p \preceq_3 q$ if and only if $p,q$ is a decreasing pair and $p$ has a smaller $x$-coordinate than $q$,
\item
$p \preceq_4 q$ if and only if $q \preceq_3 p$.
\end{itemize}

All the sets $\cL'_i$, as well as $\cL_{0,decr}$ and $\cL_{0,incr}$, can be defined with respect to the layers according to $(X,\preceq_j)$ for all $j=1,2,3,4$. We claim that for at least one of them $A_{0,0}=\emptyset$ and $|C|\neq 1$ must hold. Note first that a point in $A_{0,0}$ forms a decreasing pair with all other points in $X$ when $A_{0,0}$ is defined according to $(X,\preceq_1)$ or $(X,\preceq_2)$, while it forms an increasing pair with all other points in $X$ when $A_{0,0}$ is defined according to $(X,\preceq_3)$ or $(X,\preceq_4)$. Clearly, there cannot exist two points simultaneously with these properties. Thus we can assume that for, say, $(X,\preceq_1)$ and $(X,\preceq_2)$ the set $A_{0,0}$ is empty.

We still have to deal with the case where $|C|=1$. Let $p_1$ denote the only point in $C$. If $A_{0,-} \cup A_{0,+}$ contains other points besides $p_1$, then $u_1$ is defined and we obtain one decreasing line in $\cL_{0,decr}$. Otherwise, as $A_{0,-}$ is empty, we have $c_1=0$ and thus $\cL'_1$ does not contain any decreasing line. Therefore, if $A_{0,h}$ is not empty, then a decreasing line $\langle q, p_1\rangle$ with $q \in A_{0,h}$ contains all $\cA_0$ and thus is different from any other lines. If $A_{0,h}$ is empty, then we obtain $A_0=\{p_1\}$ meaning $p_1 \preceq q$ for any other point $q$. The same proof for $(X,\preceq_2)$ shows that we find an extra line unless there exists a point $p_2$ such that $p_2 \preceq_2 q$, or equivalently $q \preceq_1 p_2$, for any other point $q$. But then $\langle p_1,p_2 \rangle$ contains all points. Thus for one of the partial orders $(\preceq_1,X)$ or $(\preceq_2,X)$ we must have defined at least $|\cA_0 \cup \cA_1|$ distinct lines. 
\end{proof}

\begin{remark}
Note that all four partial orders are needed in the above proof. Consider the points $p_1=(x_1,y_1),p_2=(x_2,y_2),...,p_{n-1}=(x_{n-1},y_{n-1}),p=(x,y)$ such that $x<x_1<x_2<...<x_{n-1}$ and $y_1<y_2<...<y_{n-1}<y$  hold (this is the configuration mentioned in the introduction, showing that Theorem~\ref{thm:main1} is tight). Then only $(X,\preceq_4)$ will have the property that $A_{0,0}=\emptyset$ and $|C|\neq 1$.
\end{remark}

\vspace{3mm}
\noindent{\bf Acknowledgements}\\
We are grateful to Va\v{s}ek Chv\'{a}tal for his interest in our result and for many improvements that he suggested to the earlier version of this paper.

\vskip 0.3truecm

\bibliographystyle{siam}
\bibliography{ida}


\end{document}